\documentclass[10pt,reqno]{amsart}
\usepackage{amsmath, latexsym, amsfonts, amssymb,amsthm, amscd,epsfig,enumerate}
\usepackage{subfigure,color, float}

\advance\hoffset -.75cm

\usepackage[usesnames]{xcolor}

\definecolor{refkey}{gray}{.75}
\definecolor{labelkey}{gray}{.75}

\newcommand{\Z}{\mathbb Z}
\newcommand{\N}{\mathbb N}

\newcommand{\E}{\mathbb E}

\newcommand{\diff}{\mathrm{d}}

\newcommand{\pr}{\mathbb P}

\newtheorem{teo}{Theorem}[section]

\newtheorem{cor}[teo]{Corollary}
\newtheorem{rem}[teo]{Remark}
\newtheorem{pro}[teo]{Proposition}
\newtheorem{defn}[teo]{Definition}
\newtheorem{exmp}[teo]{Example}

\title
{Galton-Watson processes in varying environment and accessibility percolation}

\author[D.~Bertacchi]{Daniela Bertacchi}
\address{D.~Bertacchi, Dipartimento di Matematica e Applicazioni,
Universit\`a di Milano--Bicocca,
via Cozzi 53, 20125 Milano, Italy.}
\email{daniela.bertacchi\@@unimib.it}

\author[P.~M.~Rodriguez]{Pablo Martin Rodriguez}
\address{P.~M~.Rodriguez, Departamento de Matem\'atica Aplicada e Estat\'istica, Insituto de Ci\^encias Matem\'aticas e de Computa\c{c}\~ao, Universidade de S\~ao Paulo, Av. Trabalhador S\~ao-carlense 400, Centro, 13560-590 S\~ao Carlos, SP, Brazil.}
\email{pablor\@@icmc.usp.br}

\author[F.~Zucca]{Fabio Zucca}
\address{F.~Zucca, Dipartimento di Matematica,
Politecnico di Milano,
Piazza Leonardo da Vinci 32, 20133 Milano, Italy.}
\email{fabio.zucca\@@polimi.it}

\begin{document}

\begin{abstract}
This paper deals with branching processes in varying environment, namely, whose offspring distributions depend on the 
generations. We provide sufficient conditions for survival or extinction which rely only on the first and second moments of the
offspring distributions. These results are then applied
to branching processes in varying environment with selection where every particle has a real-valued label and labels can only increase along 
genealogical lineages;
we obtain analogous conditions for survival or extinction. These last results can be interpreted in terms of
accessibility percolation on Galton-Watson trees, which represents a relevant tool for modeling the evolution of biological populations.
\end{abstract}


\maketitle


\section{Introduction}
\label{sec:intro}

A branching process in varying environment (or \textit{BPVE}), also called \textit{time-inhomogeneous} branching process, is the generalization of the classical Galton-Watson 
process when the offspring distributions may vary according to the generations. 
The limit behaviour of these processes was firstly studied in \cite{cf:Agresti75,cf:Church71,cf:Jagers74,cf:Lindvall74}, and later in \cite{cf:CJ94,cf:Bigg}, 
among others. We refer the reader to \cite{cf:Jagers75} for a survey of earlier results about this topic
and for biological motivations. See also \cite{cf:BM08} for a recent study on the 
survival properties of these processes and on its connection with percolation theory on trees.

A natural generalization of the branching process is the branching random walk (or \textit{BRW}) where each particle is placed inside a space $X$ or, equivalently, it
is assigned a type (chosen in the space $X$).
BRWs are particularly relevant throughout the paper, since there is a natural identification of a BPVE with a BRW on the space $\N$ (see Section~\ref{sec:BPVE} 
for details). The case when the space $X$ is at most countable is well studied and understood in both continuous time and discrete time:
we refer the reader, for details and results on BRWs
to \cite{cf:BZ, cf:BZ3, cf:BZ15, cf:PemStac1} (continuous time), 
\cite{cf:Hautp16, cf:Bul15, cf:Bul15-2, cf:GMPV09, cf:Hautp12, cf:Hautp13, cf:MachadoMenshikovPopov, cf:MP03}
(discrete time); 
see also \cite{cf:BZ4} for a survey on the topic. Examples of BRWs with countable 
space $X$ (along with some variants) and their biological applications are presented in \cite[Ch.7]{cf:KA02}.

The case when either the space $X$ is uncountable or there is a non-trivial interaction among the particles is less understood and there is not  
a well-established systematic theory available
and, in general, different processes have to be studied with different tools
(see for instance \cite{cf:BPZ}).
 As far as we know, only a small number of papers are devoted to BRWs where the space $X$ is an uncountable set. 
 One example of such a process is proposed in \cite{cf:BLSW91}, where the positions of the particles are interpreted
 as types (\textit{reproductive prowess} to be exact) and it is assumed that a child is likely to be weaker (in some way) than its parent 
 and children who are too weak will not reproduce; the authors obtain conditions for survival on a family line.
 Another example of model with uncountably many types is \cite{cf:Schin14}, where type is the fitness of the individual.

The purpose of this paper is twofold. On the one hand, we provide conditions for survival and extinction of a BPVE.
On the other hand we apply these results to branching processes in varying environment with selection  (or \textit{BPWS}).
We obtain sufficient conditions for extinction or survival of a BPVE using a generating function approach.
In particular, we exploit the fact that survival is equivalent to the existence of a nontrivial fixed point of the generating
function of the process (Theorem~\ref{th:survival}).
We provide a sufficient condition for almost sure extinction, which involves only
the sequence of the first moments of the reproduction laws (Proposition~\ref{th:main0ext}).
Note that the conditions for survival cannot rely only on the first moments. 
Indeed given any sequence of first moments it is possible to construct a corresponding BPVE 
which dies out almost surely (see Example~\ref{exmp:largem_nextinction}).
The strategy to prove survival is to study the associated BRW on $\N$
and show that it is sufficient to control the ratio between the second moment and the product of the first moments 
of the reproduction laws.
In the second part of the paper, we study the behaviour
of a general class of branching processes in varying environment with selection (or \textit{BPWS}) which are, actually,
BRWs in varying environment 
on an uncountable space.
This class of processes 
is obtained by associating a (random) real value to each new individual, say a fitness, 
and by assuming that only those children who have a fitness greater than its parent may survive and reproduce. We shall 
see that the BPWS 
is related to the \textit{accessibility percolation model} on regular trees introduced in \cite{cf:NK},
and recently studied on spherically symmetric trees in \cite{cf:CGR14} (see Section~\ref{sec:BPWS}). 
This BRWS is useful for modeling the evolution of species (for similar models see for instance \cite{cf:GMS11,cf:GMS13, cf:LS09}).


Here is the outline of the paper.
In Section~\ref{sec:BPVE} we introduce the notion of BPVE and we describe the identification between a BPVE and a BRW on $\N$.
By means of this identification, we translate a characterization of survival for BRWs (Theorem~\ref{th:survival}) 
into a similar result for BPVEs (Proposition~\ref{pro:BPVE-BRW}) which can be applied 
to obtain
sufficient conditions for survival for BPVEs (Theorem~\ref{th:main0} and Corollary~\ref{cor:main0}). 
At the end of the section we compare these conditions to other results
found in the literature, while in Example~\ref{exm:continuous} we describe some explicit families of
offspring distributions to which Theorem~\ref{th:main0} applies.  
A sufficient condition for almost sure extinction is given in Proposition~\ref{th:main0ext}.
Unlike the time-homogeneous BP where, provided that the probability of having exactly one child is strictly smaller than one
then (almost sure) extinction is equivalent to having an average number of children not stricly larger than one, in the time-inhomogeneous case 
slightly counterintuitive situations occur. Indeed, denote by $m_n$ the average number of children of a particle of the $n$th generation;
on the one hand when $m_n<1$ for all $n \in \N$ there might be survival (see Example~\ref{exmp:less1surv}) and
on the other hand for any sequence of first moments $\{m_n\}_{n\in \N}$ (even when $m_n\to+\infty$) there are examples of almost sure extinction 
(see Example~\ref{exmp:largem_nextinction}).

Section~\ref{sec:BPWS} is devoted to the definition of a generic BPWS and its connection with the accessibility percolation model.
A condition for the extinction of a BPWS is given in Proposition~\ref{pro:extinction}, while the main condition for survival is given
in Theorem~\ref{th:main1}.

\section{Branching processes in varying environment}
\label{sec:BPVE}

\subsection{Basic definitions} \label{subsec:basicBPVE}

We begin by defining a \textit{branching process in varying environment} or \textit{BPVE}. The process starts with one particle at time $0$
(this is the $0$th generation).
The random number of particles generated by each particle in the $n$th generation has generating function
$\Phi_n(z):=\sum_{i=0}^{+\infty} \rho_n(i) z^i$ and we define a sequence of random variables $\{W_n\}_{n \in \mathbb{N}}$ by
$\pr(W_n=i):=\rho_n(i)$. Thus, $W_n$ represents the ``typical'' random number of children of a particle in
the $n$th generation; all the particles behave independently.

More formally, the BPVE is the stochastic process $\left\{Z_n\right\}_{n\in \mathbb{N}}$ such that
$$Z_{n+1}:=\sum_{j=1}^{Z_{n}}W_{n,j}, \qquad n \ge 0$$
where $Z_n$
is the number of particles in the $n$th generation, $Z_0$ is the initial state ($Z_0=1$ in our case)
and 
$\left\{W_{n,j}\right\}_{j\geq 1, n \ge 0}$ is a family of independent variables such that
$\left\{W_{n,j}\right\}_{j\geq 1}$ are identically distributed copies of $W_{n}$. 
As usual, we say that the BPVE becomes extinct almost surely if $p_e:=\pr \big (\bigcup_{n\geq 1}\{Z_n =0\} \big ) =1$;
 otherwise, we say that it survives
with positive probability (``almost surely'' and ``with positive probability'' will often be tacitly understood). 
If we define $H_0(z):=z$ for all $z \in [0,1]$ and, recursively, $H_{n+1}:= H_n \circ \Phi_{n}$, it is not difficult
to show that $H_n(0)$ is the probability that the population is extinct at time $n$; in particular $H_n(0) \uparrow p_e$ as $n \to +\infty$.
The probability of extinction is monotone with respect to $\{\Phi_n\}_{n \in \mathbb{N}}$, meaning that, if
$\Phi_n \ge \overline \Phi_n$ (where $\{\overline \Phi_n\}_{n \in \mathbb{N}}$ is the sequence of generating functions related to another BPVE with 
extinction probability $\bar p_e$), then by induction 
$H_n \ge \overline H_n$ and thus
$p_e \ge \bar p_e$. 

In order to avoid trivial situations we require that $\Phi_n(0)<1$ for all $n \in \mathbb{N}$, that is, there is always a nonzero probability
of having at least one child for a particle in any generation. This implies that there is always a positive probability
of finding descendants in the $n$th generation for any given $n$, that is, $H_n(0)<1$ for all $n \in \N$.


The main idea behind our results is the interpretation of a BPVE as a particular case of 
branching random walk. 
Indeed in a branching process all the particles are 
indistinguishable. In a branching random walk, on the other hand, particles  live
on a spatial structure and are thus characterized by their position (which can also be interpreted
as their \textit{type}).
More precisely, given a BPVE, we associate a BRW by considering the time variable $n$ as a spatial variable.

 
 A discrete-time BRW on an at most countable set $X$
 is a stochastic process $\{\eta_n\}_{n \in \N}$, where 
 $\eta_n(x)$ represents the number of particles alive
at $x \in X$ at time $n$. More formally,
consider a family $\nu=\{\nu_x\}_{x \in X}$
of probability measures 
on the (countable) measurable space $(S_X,2^{S_X})$
where $S_X:=\big \{f:X \to \N\colon \sum_{y \in X} f(y)<+\infty \big \}$.
To obtain generation $n+1$ from generation $n$ we proceed as follows:
a particle at site $x\in X$ lives one unit of time,
then a function $f \in S_X$ is chosen at random according to the law $\nu_x$
and the original particle is replaced by $f(y)$ particles at
$y$, for all $y \in X$; this is done independently for all particles of
generation $n$.
Note that the choice of $f$ simultaneously assigns the total number of children
and the location where they will live.


We consider initial configurations with only one particle placed at a fixed site $x$:
let $\pr^{\delta_{x}}$ be the law of this process. 

\begin{defn}\label{def:survival}

The BRW \textsl{survives (globally) with positive probability} starting from $x$ if
$
\bar {\mathbf{q}}(x):=
1- \pr^{\delta_x} \Big (\sum_{w \in X} \eta_n(w)>0, \forall n \in N \Big )
<1.
$

\end{defn}
We remark here that a globally surviving BRW can also survive locally, meaning that 
with positive probability there will be infinitely many returns to the starting location.
Since here we are just interested in the global survival, we refer the reader to \cite{cf:BZ2, cf:Z1}
for details.

Global survival can be characterized by using a generating function associated to the BRW:
namely the function $G:[0,1]^X \to [0,1]^X$ where, 
for all ${\mathbf{q}} \in [0,1]^X$, $G({\mathbf{q}}) \in [0,1]^X$ is the following weighted sum of (finite) products
\[ 
G({\mathbf{q}}|x):= \int_{S_X} \nu_x(\diff f) \prod_{y \in X} {\mathbf{q}}(y)^{f(y)}=
\sum_{f \in S_X} \nu_x(f) \prod_{y \in X} {\mathbf{q}}(y)^{f(y)},
\] 
$G({\mathbf{q}}|x)$ being the $x$ coordinate of $G({\mathbf{q}})$.

Note that $[0,1]^X$ is a partially ordered set where $\mathbf{q} \ge \mathbf{z}$ if and only if 
 $\mathbf{q}(x) \ge \mathbf{z}(x)$ for all $x \in X$; clearly $\mathbf{q} > \mathbf{z}$ stands for
 ``$\mathbf{q} \ge \mathbf{z}$ and $\mathbf{q}(x) > \mathbf{z}(x)$ for some $x \in X$''.
The function $G$ is nondecreasing and continuous with respect to the product topology on $[0,1]^X$ and
the family $\{\nu_x\}_{x \in X}$ is uniquely determined by this generating function. 

It is easy to show (see for instance \cite[Corollary 2.2]{cf:BZ2} or the proof of Theorem~\ref{th:survival}) that $\bar {\mathbf{q}}$
is the smallest solution of $G(\mathbf{q}) \le \mathbf{q}$ in $[0,1]^X$, in particular it is the smallest
fixed point of $G$ in $[0,1]^X$, that is $G(\bar {\mathbf{q}})=\bar {\mathbf{q}}$.

Define the first moments $m_{xy}:=\sum_{f \in S_X} f(y)\nu_x(f)$; denote by $m_{xy}^{(0)}:=\delta_{xy}$ and
$m_{xy}^{(n+1)}:=\sum_{w \in X} m^{(n)}_{xw}m_{wy}$ for all $n \in \N$ 
(clearly, by using $+\infty\cdot 0:=0$ and $+\infty \cdot x := +\infty$ for all $x>0$, we have
$m_{xy}^{(n)} \in [0,+\infty]$ for all $n \in N$, $x,y \in X$). Given $\mathbf{v} \in [0,+\infty]^X$ we define
$Mv \in  [0,+\infty]^X$ by $Mv(x):=\sum_{w \in X} m_{xw} v(w)$; clearly $M^nv(x)=\sum_{w \in X} m^{(n)}_{xw} v(w)$.
The following theorem characterizes
global survival; it appears, in different flavors, in \cite[Theorem 4.1]{cf:Z1} or \cite[Theorem 3.1]{cf:BZ14-SLS} and 
it is based on \cite[Proposition 2.1]{cf:BZ2}. Unlike those reults, here we remove the requirement that  
$\sum_{y \in X} m_{xy}$ is bounded; hence we write the proof which is slightly different from
the ones in the above cited papers. Henceforth,  by $\mathbf{0}, \mathbf{1} \in[0,1]^X$ we mean the constant functions
$\mathbf{0}(x):=0$, $\mathbf{1}(x):=1$
for all $x \in X$; note that $G(\mathbf{1})=\mathbf{1}$.

\begin{teo}\label{th:survival}
 Consider a BRW and a fixed $x \in X$. The following statements are equivalent:
 \begin{enumerate}
  \item $\bar {\mathbf{q}}(x)<1$ (i.e.~there is global survival starting from $x$);
  \item there exists
${\mathbf{q}}\in [0,1]^X$ such that ${\mathbf{q}}(x)<1$ and $G(\mathbf{q}) \le \mathbf{q}$ 
(i.e.~$G({\mathbf{q}}|y) \le {\mathbf{q}}(y)$, for all $y \in X$);
  \item there exists
${\mathbf{q}}\in [0,1]^X$ such that ${\mathbf{q}}(x)<1$ and $G(\mathbf{q}) = \mathbf{q}$ 
(i.e.~$G({\mathbf{q}}|y) = {\mathbf{q}}(y)$, for all $y \in X$).
 \end{enumerate}
 If $\mathbf{q}$ satisfies either $(2)$ or $(3)$, then $\mathbf{q} \ge \bar {\mathbf{q}}$.
 Moreover, 
 global survival starting from $x$ implies that $\liminf_{n \to +\infty} \sum_{y \in X} m_{xy}^{(n)}>0$. 
 \end{teo}

\begin{proof}
Consider the sequence $\{\mathbf{q}_n\}_{n \in \N}$ defined as
\[
 \begin{cases}
\mathbf{q}_0:= \mathbf{0}  &\\
\mathbf{q}_{n+1}:=G(\mathbf{q}_n), & \forall n \in \N
 \end{cases}
\]
clearly $\mathbf{q}_{n}(x)$ is the probability that the process, which
starts with one particle at $x$ at time $0$, has no particles at time $n$ .
Moreover $\mathbf{q}_{n}$ converges pointwise to $\bar {\mathbf{q}}$ (that is, with respect to the product topology). 
By the continuity of $G$ we have $\bar {\mathbf{q}}=G(\bar {\mathbf{q}})$.

Now $(1) \Longrightarrow (3) \Longrightarrow (2)$ are trivial. Assume $(2)$;
by induction on $n$ we have that $\mathbf{q}_n \le \mathbf{q}$; indeed $\mathbf{0} \le \mathbf{q}$ and, 
since $G$ is nondecreasing, 
$\mathbf{q}_{n+1}=G(\mathbf{q}_n) \le G(\mathbf{q}) \le \mathbf{q}$. By taking the limit as $n \to +\infty$ we have 
$\bar {\mathbf{q}}\le \mathbf{q}$. This implies $\bar {\mathbf{q}}(x)\le \mathbf{q}(x)<1$; thus $(1)$ is proven.

We are left to prove that, say, $(1)$ implies $\liminf_{n \to +\infty} \sum_{y \in X} m_{xy}^{(n)}>0$.
To this aim consider a realization $\{\eta_n\}_{n \in \N}$ of the BRW and denote by $\mathbb{E}^x$ the expectation with respect to 
 $\pr^{\delta_x}$. If $\pr^{\delta_x}(S)>0$ where $S:=\{\sum_{y \in X} \eta_n(y) >0, \ \forall n \in \N\}$ then, since
 $\sum_{y \in X} \eta_n(y) \ge 1$ on $S$, we have
 \[
  \sum_{y \in X} m^{(n)}_{xy} = \mathbb{E}^x \Big [\sum_{y \in X} \eta_n(y) \Big ] \ge \mathbb{E}^x \Big [\sum_{y \in X} \eta_n(y) \Big | S \Big ] \pr^{\delta_x}(S) 
  \ge  \pr^{\delta_x}(S) >0.
 \]
 This implies $\inf_{n \in \N} \sum_{y \in X} m_{xy}^{(n)}>0$ which is equivalent to
 $\liminf_{n \to +\infty} \sum_{y \in X} m_{xy}^{(n)}>0$ 
(since $\sum_{y \in X} 
 m_{xy}^{(n)}=0$ for some $n$ implies the same equality for all subsequent values of $n$).
 \end{proof}
%
%

Given a BPVE with a sequence $\{\Phi_n\}_{n \in \mathbb{N}}$ of generating functions
we can construct the associated BRW on $\mathbb{N}$ as follows: each particle at $n \in \mathbb{N}$ has a random number of children according
to the $n$th generation law of the BPVE and they are all placed at $n+1$. This is a reducible BRW whose generating function 
$G:[0,1]^\mathbb{N} \to [0,1]^\mathbb{N}$ satisfies
\begin{equation}\label{eq:associated}
 G(\mathbf{q}|n):=\Phi_n(\mathbf{q}(n+1)), \qquad \forall \mathbf{q} \in [0,1]^\mathbb{N}
\end{equation}
(note that 
the same identification holds in general for a BRW in varying environment (that is, time-inhomogeneous BRW) on $X$ and a time-homogeneous
BRW on $X \times \N$).
Applying Theorem~\ref{th:survival} to the BRW associated to the BPVE we have the following characterization of survival for the
BPVE.

\begin{pro}\label{pro:BPVE-BRW}
Consider a BPVE and its sequence $\{\Phi_n\}_{n \in \mathbb{N}}$ of generating functions.
 There is survival for the process if and only if there exists $\mathbf{q} \in  [0,1]^\mathbb{N}$, $n_0 \in \mathbb{N}$ such that 
 $\mathbf{q}(n_0)<1$ and $\Phi_n(\mathbf{q}(n+1)) \le \mathbf{q}(n)$ for all $n \ge n_0$.
\end{pro}

\begin{proof}
According to Theorem~\ref{th:survival} the associated BRW survives globally if and only if there exists $\mathbf{q} \in  [0,1]^\mathbb{N}$
such that $G(\mathbf{q}) \le \mathbf{q}$ and $\mathbf{q}(n)<1$ for some $n \in \mathbb{N}$ (that is, $\mathbf{q}< \mathbf{1}$). 
By equation~\eqref{eq:associated} the condition is equivalent to  
$\Phi_n(\mathbf{q}(n+1)) \le \mathbf{q}(n)$ for all $n \ge n_0$ and
$\mathbf{q}(n_0)<1$ for some $n_0$; 
indeed we can always define $\mathbf{q}(i)=1$ for all $i=0,1, \ldots, n_0-1$ and we have $\Phi_n(\mathbf{q}(n+1)) \le \mathbf{q}(n)$ for all $n \in \mathbb{N}$.
This implies survival starting from the $n_0$th generation.

However, since $\Phi_n(0)<1$ for all $n \in \mathbb{N}$,  there is a positive probability for the BPVE to survive up to the 
$n_0$th generation (for every fixed $n_0 \in \mathbb{N}$). 
Thus, 
there is survival starting from the $0$th generation 
if and only if there is survival 
starting from the $n_0$th generation. 
\end{proof}

\subsection{Main results}\label{subsec:mainBPVE}

We consider a BPVE and its sequence $\{\Phi_n\}_{n \in \mathbb{N}}$ of generating functions. 
We denote by 
 $m_n$  the first moment $\E[W_n]=\Phi_n^\prime(1)$ of the reproduction law of the $n$th generation.
The first results is a sufficient condition for extinction (see
Example~\ref{exmp:less1surv} for an application). 
To avoid trivial situations we assume henceforth that $m_n>0$ for all $n \in \N$.
\begin{pro}\label{th:main0ext}
 If $\inf_{n \in \N} \prod_{i=0}^n  m_i =0$, then the BPVE dies out.  
\end{pro}

\begin{proof}
Since
$m_n>0$ for all $n \in \N$ then $\inf_{n \in \N} \prod_{i=0}^n  m_i =0$ if and only if
$\liminf_{n \to +\infty} \prod_{i=0}^n  m_i =0$.
Extinction follows from Theorem~\ref{th:survival}, since the expected number of children
in $n$ steps is $\prod_{i=0}^{n} m_i$.
We show that the thesis can also be derived by Proposition~\ref{pro:BPVE-BRW}.
Indeed, suppose by contradiction that there is survival. Thus,
there exists $\mathbf{q} \in  [0,1]^\mathbb{N}$, $n_0 \in \mathbb{N}$ such that 
 $\mathbf{q}(n_0)<1$ and $\Phi_n(\mathbf{q}(n+1)) \le \mathbf{q}(n)$ for all $n \ge n_0$.
 By convexity, since $\Phi^\prime_n(1)=m_n$, we have 
 \[
\Phi_n(z)
  \ge 1-m_n +m_n z, \qquad \forall z \in [0,1], 
 \]
 hence $1-\mathbf{q}(n) \le m_n (1-\mathbf{q}(n+1))$ for all $n \ge n_0$.
 Thus, by induction on $j$,
 \[
  0<1-\mathbf{q}(n_0) \le (1- \mathbf{q}(n_0+j+1)) \prod_{i=n_0}^{n_0+j} m_i \le   
\prod_{i=n_0}^{n_0+j} m_i, 
  \]
  for all $j \ge 0$.
  Since $\prod_{i=0}^{n_0-1} m_i >0$ then we have $\inf_{n \in \N} \prod_{i=0}^{n} m_i>0$.
\end{proof}
Observe that a sufficient condition for extinction is 
$\limsup_{n \to +\infty} m_n <1$: Proposition~\ref{th:main0ext} applies and the BPVE dies out.

We denote now by 
 $m_n^{(2)}$  the  second moment $\E[W_n^2]$ of the reproduction law of the $n$th generation;
 henceforth we suppose that this moment is finite for every sufficiently large $n$.
 Note that 
 $m_n^{(2)}=\Phi_n^{\prime\prime}(1)+m_n$.
Theorem~\ref{th:main0} and Corollary~\ref{cor:main0} provide sufficient conditions for survival
(see Example~\ref{exmp:largem_nextinction} for an application). 

\begin{teo}\label{th:main0}
Consider a BPVE such that $m^{(2)}_n<+\infty$ for every sufficiently large $n$.
Then, for every $n \in \N$, the following statements are equivalent:
\begin{enumerate}
 \item 
 \[
 \begin{cases}
  \sum_{j=n}^{+\infty} \frac{m^{(2)}_j-m_j}{m_j} \Big ( \prod_{i=n}^j m_i \Big )^{-1} <+\infty\\
  \inf_{j \in \N} \prod_{i=0}^j m_i >0\\
  \end{cases}
 \]
 \item
 \[
 \lim_{k \to +\infty} \Big [\Big ( \prod_{i=n}^{n+k} m_i \Big )^{-1} +   \sum_{j=n}^{n+k} \frac{m^{(2)}_j-m_j}{m_j} \Big ( \prod_{i=n}^j m_i \Big )^{-1} \Big ] <+\infty.
 \]
\end{enumerate}
Moreover, if one of these conditions holds for some $n$ then the BPVE survives. 
\end{teo}

%
%
Note that if $(1)$ (resp.~$(2)$) holds for some $n=n_0$ than it holds for every $n \ge n_0$.

\begin{proof}
The idea is to construct a solution $\mathbf{q}$ as in Proposition~\ref{pro:BPVE-BRW}. 
To this aim, we make use of an upper bound due to \cite{cf:Agresti74}: 
$\Phi_n(x) \le f_n(x)$ for all $x \in [0,1]$ where $f_n(x):=1-b_n/(1-c_n)+b_nx/(1-c_nx)$ with
$b_n:=m_n^3/(m^{(2)}_n)^2$ and $c_n:=(m_n^{(2)}-m_n)/m_n^{(2)}$. In particular,
by defining $1/0:=+\infty$ and $1/+\infty:=0$, we have
\[
 f_n(x)=1-\frac{1}{\xi_n(1/(1-x))}
\]
for all $x \in [0,1]$ where
\[
 \xi_n(s):=
 \begin{cases}
s/m_n+  (m_n^{(2)}-m_n)/m_n^2 & s \in \mathbb{R}\\
+\infty & x=+\infty.
 \end{cases}
\]
Define $S_{n,l}:= \sum_{j=n}^{l} \frac{m^{(2)}_j-m_j}{m_j} \big ( \prod_{i=n}^j m_i \big )^{-1}$ and 
$\beta_{n,l}:=S_{n,l}+\big( \prod_{i=n}^{l} m_i \big )^{-1}$ (for all $0 \le n \le l$). Note that
$S_{n,l}$ is nondecreasing with respect to $l$: indeed,
$W_j$ is an integer-valued 
random variable, whence $m^{(2)}_j \ge m_j$. Now we prove that also $l \mapsto \beta_{n,l}$ is nondecreasing.
Indeed it is trivial to show that $\beta_{n,l+1}-\beta_{n,l}=(m^{(2)}_{l+1}/m^{2}_{l+1}-1) ( \prod_{i=n}^l m_i \big )^{-1} \ge 0$.
Observe that, for every $n \in \N$,  
$\inf_{j \ge n} \prod_{i=n}^j  m_i >0$ if and only if $\inf_{j \in \N} \prod_{i=0}^j m_i >0$; moreover,
$\prod_{i=0}^j m_i$ and $\prod_{i=n}^j m_i$ have the same behaviour as $j \to +\infty$.

$(1) \Longrightarrow (2)$. Indeed if $S_{n,n+k}$ converges as $k \to +\infty$ and 
$\beta_{n,n+k} - S_{n,n+k}$ is bounded from above with respect to $k$, then $\beta_{n,n+k}$ is bounded from above with respect to $k$, thus
the convergence follows from the monotonicity.

$(2) \Longrightarrow (1)$. Clearly $S_{n,n+k} \le \beta_{n,n+k}$; if $\beta_{n,n+k}$ converges as $k \to +\infty$, 
since $S_{n,n+k}$ is non decreasing then it converges, thus $\beta_{n,n+k}-S_{n,n+k}$ is bounded from above with respect to $k$.

$(2) \Longrightarrow$ survival. Denote $\lim_{k \to +\infty} \beta_{n,n+k}$ by $b_n$.
Now we prove that: (a)
$\mathbf{q}(n):=1-1/b_n \in [0,1]$ for all $n \in \N$, (b) $\mathbf{q}(n)<1$ for some $n$ and (c) $\mathbf{q}$ is a solution of 
$\Phi_n(\mathbf{q}(n+1)) \le \mathbf{q}(n)$.
Thus Proposition~\ref{pro:BPVE-BRW} applies.

Clearly $b_n \ge \beta_{n,n} = 1/m_n + (m_n^{(2)}-m_n)/m_n^2=m_n^{(2)}/m_n^2 \ge 1$, whence $\mathbf{q}(n) \in [0,1]$ and (a) is proved. 
Moreover $(2)$ implies $1-1/b_n<1$, that is, (b). 
To prove (c) it suffices to show that
$\mathbf{q}(n)=f_n(\mathbf{q}(n+1)) \ge \Phi_n(\mathbf{q}(n+1))$. To this aim
we show that $\xi_n(b_{n+1})=b_n$. Indeed, by using the continuity 
of $\xi_n$,
\[
\begin{split}
 \xi_n(b_{n+1})&= \lim_{k \to +\infty} (\beta_{n+1,n+k+1}/m_n+ (m_n^{(2)}-m_n)/m_n^2)\\
 &= 
 \lim_{k \to +\infty} \Big [ \Big ( \prod_{i=n}^{n+k+1} m_i \Big )^{-1} +   \sum_{j=n+1}^{n+k+1} \frac{m^{(2)}_j-m_j}{m_j} \Big ( \prod_{i=n}^j m_i \Big )^{-1} 
 + (m_n^{(2)}-m_n)/m_n^2) \Big ]\\
 &=\lim_{k \to +\infty} \beta_{n,n+k+1}=b_n.\\
 \end{split}
\]
\end{proof}

The conditions in Theorem~\ref{th:main0} are implied by other conditions which are easier to check, as the
following corollary shows.

\begin{cor}\label{cor:main0}
 Consider a BPVE such that  $m^{(2)}_n<+\infty$ for every sufficiently large $n$.
 If one of the following holds:
 \begin{enumerate}
  \item $\sum_{j=n}^{+\infty} \frac{m^{(2)}_j}{m_j} \Big ( \prod_{i=n}^j m_i \Big )^{-1} <+\infty$ for some $n \in \N$;
  \item $\limsup_{n \to +\infty} \sqrt[n]{m^{(2)}_n/{m^2_n}} < \liminf_{n \to +\infty} \sqrt[n]{\prod_{i=0}^{n-1} m_i }$;
  \item there exists a function $g:\N\to [1,+\infty)$ such that $m^{(2)}_n/{m^2_n} \le g(n)$ for every sufficiently large $n$ and 
  $\limsup_{n \to +\infty} g(n+1)/g(n) < \liminf_{n \to +\infty} \sqrt[n]{\prod_{i=0}^{n-1} m_i }$;
  \item $\lim_{n \to +\infty} m_n=+\infty $ and there exists $M,k \ge 1$ such that $m_n^{(2)}/m_n^2 \le k M^n$ for all sufficiently large $n\in \mathbb{N}$;
  \end{enumerate}
then the BPVE survives. 
\end{cor}

\begin{proof} 
It is enough to prove that $(4) \Longrightarrow (3) \Longrightarrow (2) \Longrightarrow (1) \Longrightarrow$ survival.

\noindent
 $(1) \Longrightarrow$ survival. Since ${m^{(2)}_j}/{m_j} \ge 1$ then $\sum_{j=n}^{+\infty} \frac{m^{(2)}_j}{m_j} \big ( \prod_{i=n}^{j} m_i \big )^{-1} <+\infty$ implies both
 $\lim_{n \to +\infty} \prod_{i=0}^n m_i =+\infty$ and $\sum_{j=n}^{+\infty} \frac{(m^{(2)}_j-m_j)}{m_j} \big ( \prod_{i=n}^j m_i \big )^{-1} <+\infty$
 whence condition $(1)$ of Theorem~\ref{th:main0} holds and the survival follows. 

\noindent
 $(2) \Longrightarrow (1)$. It follow easily from Cauchy's Root Test. 

\noindent
 $(3) \Longrightarrow (2)$. We observe that since $g:\N\to [1,+\infty)$ then $\limsup_{n \to +\infty} g(n+1)/g(n) \ge 1$. 
 For every $\varepsilon >0 $ define $K_\varepsilon:=\sup_{n} \big ( \prod_{i=0}^{n-1} g(i+1)/g(i) \big ) / 
(\limsup_{n \to +\infty} g(n+1)/g(n)+\varepsilon)^n <+\infty$; then
\[
 m^{(2)}_n/{m^2_n} \le g(n)= g(0) \prod_{i=0}^{n-1} g(i+1)/g(i) \le g(0) K_\varepsilon (\limsup_{n \to +\infty} g(n+1)/g(n)+\varepsilon)^n
\]
 which implies $\limsup_{n \to +\infty} \sqrt[n]{m^{(2)}_n/{m^2_n}} \le \limsup_{n \to +\infty} g(n+1)/g(n) +\varepsilon$. This
 can be done for every $\varepsilon >0$, hence
 $$\limsup_{n \to +\infty} \sqrt[n]{m^{(2)}_n/{m^2_n}}\le \limsup_{n \to +\infty} g(n+1)/g(n)<\liminf_{n \to +\infty} \sqrt[n]{\prod_{i=0}^{n-1} m_i}.$$

\noindent
 $(4) \Longrightarrow (3)$. 
 It is enough to choose $g(n):=k M^n$.

\end{proof}

%


In the following example we consider some relevant laws for $W_n$ which satisfy the sufficient conditions
of Theorem~\ref{th:main0}.

\begin{exmp}\label{exm:continuous}
Consider the following geometric reproduction laws $\rho_n(i)= m_n^i/(1+m_n)^{i+1}$.
 This family of laws is particularly relevant since they represent the total number of children of a particle in a continuous-time branching process
 with breeding rate $m_n$ and death rate $1$. If $\sum_{n=0}^{+\infty} (\prod_{j=0}^n m_j)^{-1}<+\infty$ then the BPVE survives.
 Indeed, the generating function is $\Phi_n(z):=1/(1+m_n(1-z))$, whence 
 the average number of children is $\frac{\diff}{\diff z}\Phi_n(z)|_{z=1}=m_n$ and $m^{(2)}_n-m_n=\frac{\diff^2}{\diff z}\Phi_n(z)|_{z=1}=2m_n^2$ which
 implies $(m^{(2)}_n-m_n)/m_n^2 =2$ for all $n$. The result follows from Theorem~\ref{th:main0}.
 A partial converse holds: if $\inf_{n \in \N} m_n >0$ and $\sum_{ n=0}^{+\infty} (\prod_{j=0}^n m_j)^{-1}=+\infty$ then the BPVE goes extinct (see
 \cite[Theorem 2]{cf:Agresti75}).
 
  Besides geometric laws, other examples are: 
  Poisson laws $W_n \sim \mathcal{P}(m_n)$ where $\sum_{n=0}^{+\infty} (\prod_{j=0}^n m_j)^{-1}<+\infty$ and 
 binomial laws $W_n \sim \mathcal{B}(k_n, r_n)$
 where $\sum_{n=0}^{+\infty} (\prod_{j=0}^n k_j r_j)^{-1}<+\infty$ (remember that $m_n=k_n r_n$).
 As before, $\inf_{n \in \N} m_n >0$ and $\sum_{ n=0}^{+\infty} (\prod_{j=0}^n m_j)^{-1}=+\infty$ imply extinction.
\end{exmp}

The following two examples show that a BPVE can survive even if $m_n<1$ for all $n$, while it can die out
whatever the sequence $\{m_n\}_{n \in \N}$ (even if $\inf_{n \in \N} m_n > 1$).

\begin{exmp}\label{exmp:less1surv}
Let us consider a sequence $\{a_n\}_{n \in \N}$ such that $a_n \in (0,1)$ for all $n$. Define 
$W_n$ as a Bernoulli variable with parameter $1-a_n$. Clearly $m_n^{(2)}=m_n=1-a_n <1$ for all $n$: 
the corresponding BPVE survives with positive probability if and only if  and $\sum_{n \in \N} a_n <+\infty$.

It is well known that, since $a_n <1$ for all $n \in \N$ then $\sum_{n \in \N} a_n <+\infty$ if and only if $\prod_{n \in \N} (1-a_n)>0$.
If $\sum_{n \in \N} a_n =+\infty$ then $\prod_{n \in \N} m_n = \prod_{n \in \N} (1-a_n)=0$ and
the BPVE dies out according to Proposition~\ref{th:main0ext}.
Survival would follow analogously when $\sum_{n \in \N} a_n <+\infty$ 
by applying Theorem~\ref{th:main0}; nevertheless a direct short proof is possible. 

Indeed, consider the following relation which
 holds for every sequence of events $\{A_i\}_{i \in \N}$: 
\begin{equation}\label{eq:conditional}
 \pr \big (\bigcap_{i=0}^{+\infty} A_i \big )>0 \Longleftrightarrow 
 \begin{cases}
\pr \big (A_i^c|\bigcap_{j=0}^{i-1} A_j \big )<1, \quad \forall i \ge 0\\
\\
\sum_{i=0}^{+\infty} \pr \big (A_i^c|\bigcap_{j=0}^{i-1} A_j \big ) <+\infty
 \end{cases}
 \end{equation}
where $\pr \big (A_0^c|\bigcap_{j=0}^{-1} A_j \big ):=\pr(A_0^c)$.
 Denote by $A_n$ the event ``the BPVE survives up to time $n$''. Hence $\pr \big (A_n^c|\bigcap_{j=1}^{n-1} A_j \big )=a_{n-1}$ (for all $n \ge 1$),
$\bigcap_{i=1}^{+\infty} A_i$ is the event of survival
and the result follows from equation~\eqref{eq:conditional}.
\end{exmp}

\begin{exmp}\label{exmp:largem_nextinction}
 Consider a nonnegative sequence $\{m_n\}_{n \in \N}$ (note that even $m_n \to +\infty$ will do). Define
 $W_n$ by 
 \[
  \pr(W_n=i)=
  \begin{cases}
m_n/k_n & \textrm{if } i=k_n\\
1-m_n/k_n & \textrm{if } i=0\\
  \end{cases}
 \]
where the sequence $\{k_n\}_{n \in \N}$ of integers satisfies 
\[
 \sum_{n \in \N} (1-m_n/k_n)^{\prod_{i=0}^{n-1} k_i}=+\infty.
\]
Note that $m_n=\E[W_n]$. We show recursively that such a sequence $\{k_n\}_{n \in \N}$ exists and we claim that the corresponding BPVE
dies out almost surely.

Indeed, consider any sequence $\{a_n\}_{n \in \N}$ such that $a_n \in (0,1)$ and $\sum_{n\in \N}a_n =+\infty$
(take for instance $a_n:=\varepsilon>0$ for all $n$). The idea is to find $\{k_n\}_{n \in \N}$ in such a way that 
$(1-m_n/k_n)^{\prod_{i=0}^{n-1} k_i} \ge a_n$.
Fix $k_0 \in \N$ such that $1-m_0/k_0 \ge a_0$. Suppose we already defined $k_i$
for all $i \le n-1$; since $(1-m_n/x)^{\prod_{i=0}^{n-1} k_i} \to 1$ as $x \to +\infty$, there exists
$k_n \in \N$ such that $(1-m_n/k_n)^{\prod_{i=0}^{n-1} k_i} \ge a_n$.

Now denote as before by $A_n$ the event ``the BPVE survives up to time $n$''. Since the maximum number of individuals alive
at time $n$ is $\prod_{i=0}^{n-1} k_i$ we have
$\pr \big (A_n^c|\bigcap_{j=1}^{n-1} A_j \big ) \ge (1-m_{n-1}/k_{n-1})^{\prod_{i=0}^{n-2} k_i} \ge a_{n-1}$ for all $n \ge 1$ (where
$\prod_{i=0}^{-1} k_i:=1$). 
 The result follows again from equation~\eqref{eq:conditional}.
 
 For an explicit example, take $m_n:=2$ for all $n$, $k_0>2$ and $k_n := k_0^{2^{n-1}}$ for all $n \ge 1$. Clearly
 $\prod_{i=o}^{n-1} k_i = k_0^{2^{n-1}} =k_n$ hence $0<(1-m_n/k_n)^{\prod_{i=0}^{n-1} k_i}=(1-2/k_n)^{k_n} \to e^{-2}$ which implies
 $\min_n (1-m_n/k_n)^{\prod_{i=0}^{n-1} k_i} >0$; thus $\sum_{n \in \N} (1-m_n/k_n)^{\prod_{i=0}^{n-1} k_i}=+\infty$.
\end{exmp}

We compare our results with other conditions found in the literature (see for instance \cite{cf:Agresti75, cf:BM08, cf:Bigg}).

 Theorem~\ref{th:main0} extends \cite[Theorem 2]{cf:Agresti75}.
 In \cite[Theorem 2]{cf:Agresti75} the author gives a characterization of survival under the condition
$\sup_{j \in \N}(m^{(2)}_j-m_j)/m_n<+\infty$ while we do not need such an inequality to be satisfied.
As an example, consider $\{W_n\}_{n\in \N}$ such that 
\[
 \pr(W_n=i)=
 \begin{cases}
2/\alpha_n & i=\alpha_n\\
1-2/\alpha_n & i=0\\
 \end{cases}
\]
where $\alpha_n \in \N$, $\lim_{n \to +\infty} \alpha_n=+\infty$ and $\sum_{n \in \N} \alpha_n/2^n<+\infty$.
Clearly $m_n=2$ and $m_n^{(2)}=2 \alpha_n$ which implies $\sup_{j \in \N}(m^{(2)}_j-m_j)/m_n 
=+\infty$.
Nevertheless $\sum_{j=0}^{+\infty} \frac{m^{(2)}_j}{m_j} \Big ( \prod_{i=0}^j m_i \Big )^{-1}=\sum_{j=0}^{+\infty} \alpha_j/2^{j+1}<+\infty$
and Corollary~\ref{cor:main0} applies.


Our results imply \cite[Proposition 1.1]{cf:BM08}. Indeed the sufficient condition for extinction in \cite[Proposition 1.1]{cf:BM08}
is a consequence of Proposition~\ref{th:main0ext}. On the other hand, the sufficient condition for survival in \cite[Proposition 1.1]{cf:BM08} follows from
Corollary~\ref{cor:main0}(3) by taking $g(n):=\sup_{j \in \N} m^{(2)}_j / \inf_{j \in \N} m^2_j >0$.

Another sufficient condition for survival of a BPVE, given by \cite[Theorem 1]{cf:Bigg}, is 
the existence of a random variable $X$ with finite expected
 value such that
 \begin{equation}\label{eq:DB}
  \pr(W_n/ m_n > x) \le \pr(X >x),\qquad \forall x \ge 0, n \in \mathbb{N}.
 \end{equation}
Theorem~\ref{th:main0} and \cite[Theorem 1]{cf:Bigg} are not in general comparable.
%

More precisely,
condition \eqref{eq:DB}
does not imply the finiteness of the second moment $m^{(2)}_n=\mathbb{E}[W_n^2]$;
on the other hand, there are examples of sequences $\{W_n\}_{n \in \mathbb{N}}$, satisfying the conditions of Corollary~\ref{cor:main0}(3) 
such that condition \eqref{eq:DB} does not hold for any $X$ with finite first moment. 
Indeed, define $W_n=\alpha_n B_n$ where $B_n$ is a Bernoulli random variable with parameter $1/n^2$ and $\liminf_{n \to +\infty} \alpha_n/n^2>1$:
clearly, $m^{(2)}_n/m^2_n =\mathbb{E}[W_n^2]/\mathbb{E}[W_n]^2=n^2=:g(n)$ and $g(n+1)/g(n) \to 1$ as $n \to +\infty$ and 
Corollary~\ref{cor:main0}(3) applies, while 
$\pr(X >x)  \ge \sup_{n \in \mathbb{N}} \pr(W_n>x)=1/n^2$ for $x \in [(n-1)^2, n^2)$ which implies $\mathbb{E}[X]=\int_0^{+\infty} \pr(X>x) \diff x
=+\infty$.

A partial equivalence between condition~\eqref{eq:DB} and, say, Corollary~\ref{cor:main0}(3) 
can be obtained under the assumptions $g(n)=M \in \mathbb{R}$ for every $n \in \mathbb{N}$.
More precisely, assume that $m_n^{(k)}/m_n^k \le M$ (for some $k>1$) then, it is easy to prove condition
\eqref{eq:DB} for the random variable $X$ with the following tails
\[
 \pr(X>x)
 :=
 \begin{cases}
 1 & \textrm{if } x \le \sqrt[k]{M}\\
  \frac{M}{x^k} & \textrm{if } x > \sqrt[k]{M}.\\
 
 \end{cases}
\]
On the other hand, if condition \eqref{eq:DB} is satisfied for some $X$ and $\mathbb{E}[X^k] \le M$ then
\[
 \frac{m_n^{(k)}}{m_n^k}=\mathbb{E}[(W_n/m_n)^k]=\int_0^{+\infty} \pr(W_n/m_n>\sqrt[k]{x}) \diff x \le \int_0^{+\infty} \pr(X>\sqrt[k]{x})
 \diff x 
 =\mathbb{E}[X^k] \le M.
\]


\section{Branching processes with selection and accessibility percolation}
\label{sec:BPWS}

\subsection{Basic definitions} \label{subsec:basicBPWS}

Given a BPVE, each individual can be assigned a label; this label can be interpreted as a position, a type or a fitness.
We assume that the label is assigned at birth independently for each individual, according to a non-atomic measure $\mu$ on $\mathbb{R}$
(that is, $x \mapsto \mu(-\infty, x)$ is a continuous map). 

By using this label we define a selection mechanism as follows:  
all children of a particle living at $x \in \mathbb{R}$ survive
if and only if they are placed in  the interval $[x, +\infty)$.
This is a Bernoulli-type selection, meaning that every child 
survives (independently) with probability $\mu(x,+\infty)$. Hence, elementary computations show that the generating function after selection of number of 
children of a particle at $x$ of generation $n$  
is $G_{n,x}(z):=\Phi_n \big (z \mu(x,+\infty)+1-\mu(x,+\infty) \big )$. The expected number of children, before selection, of a particle in generation $n$ is
$m_n=\E[W_n]=\Phi_n^\prime(1)= \sum_{i \in \mathbb{N}} i \rho_n(i)$; after selection, given the position $x$ of the parent, is clearly
$G_{n,x}^\prime(1)= \Phi_n^\prime(1) \mu(x,+\infty)=m_n  \mu(x,+\infty)$. We call this process \textit{Branching Process in varying environment with selection} or
\textit{BPWS}. Note that a BPWS 
is a particular case of time-inhomogeneous BRW on an uncountable space.  

One graphical way to construct the BPWS is to generate the Galton-Watson tree of the progeny of the BPVE
before selection (starting with one individual represented by the root of the tree)
and to associate independently to every vertex $v$ a random variable $X_v \sim \mu$.
Clearly the BPWS erases all the subtrees branching from a vertex $v^\prime$ such that $X_{v^\prime}<X_v$, where
$v$ is the parent of $v^\prime$.

This process can be seen as a particular case of a more general family of processes, namely the \textit{accessibility percolation model}, 
introduced in \cite{cf:NK} and inspired by evolutionary biology questions. 
In this model one considers a graph $G=(\mathcal{V},\mathcal{E})$, and associates to each vertex $v\in \mathcal{V}$ a random variable $X_v$ 
belonging to a sequence of independent identically distributed, continuous random variables. The main question of interest 
is the existence of a self-avoiding path of vertices $\{v_i\}_{i \in \N}$ crossing the entire graph, 
such that $X_{v_i} \le X_{v_{i+1}}$ for all $i \in \N$. 
Such a path is called \textit{accessibility path} and the existence of at least one of them, with positive probability, is called 
\textit{accessibility percolation}. 
This question has been addressed mainly on regular trees and hypercubes in \cite{cf:BBS,cf:NK,cf:RZ,cf:SK}. 

In order to study the behaviour of a BPWS,
we denote by $\mathcal{A}_n$ the random set of positions of the particles of generation $n$;
hence, the size of the population is
$N_n:=\# \mathcal{A}_n$ ($\#$ represents the cardinality of a set) almost surely. 

\begin{defn}\label{def:extinction}
 \begin{enumerate}
 \item We define the probability of \textit{local extinction} in $I \subseteq \mathbb{R}$ starting from $x$ by 
 $\pr(\liminf_{n \to +\infty} \{\mathcal{A}_n \cap I=\emptyset\}|\mathcal{A}_0=\{x\})$. We say that there is \textit{local survival} when
 this probability is strictly smaller than 1.
   \item We say that there is \textit{global extinction} starting from $x$ if and only if there is local extinction in $\mathbb{R}$
   starting from $x$.  There is global survival  starting from $x$ if and only if 
   $\pr(\mathcal{A}_n \not= \emptyset, \forall n \in \mathbb{N}|\mathcal{A}_0=\{x\}) 
   \equiv \pr(N_n>0, \forall n \in \mathbb{N}|\mathcal{A}_0=\{x\})>0$.
 \end{enumerate}
\end{defn}

Clearly, given a BPWS global survival is equivalent to accessibility percolation on its (infinite) Galton-Watson tree.
It is clear from the definition that local survival implies global survival.
We note that the progeny of a particle living at $x$ is located in $[x,+\infty)$; moreover, if we are interested in \textit{local survival}, that is,
the survival of the progeny in an interval $(a,b)$, we can disregard (or ``kill'') all particles placed in $[b,+\infty)$. 
Moreover, by using a coupling argument, it is easy to see that the probability of local extinction is nondecreasing with respect to $x \in \mathbb{R}$.

Sometimes it is useful to consider the position of the leftmost particle which we denote by
$l_n:=\min \mathcal{A}_n$ (where $\min(\emptyset):=+\infty$). 
By the nature of the selection process and the fact that $\mu$ is non-atomic, $\{l_n\}_{n \in \mathbb{N}}$ is a strictly increasing
random sequence almost surely. 
%
%
Given any measurable set $I$, if $\mu(I)=0$ there is local extinction in $I$. In general 
there is local survival in $I$ starting from $x$ if and only if 
 $\pr(\mu((\lim_n l_n, +\infty) \cap \mathrm{co(I)})>0)>0$
 (where $\mathrm{co}(I)$ is the \textit{essential} convex hull of $I$, that is the smallest interval $J$ such that $\mu(I\setminus J)=0$). 
 Indeed no contribution to $\mathrm{co}(I)$ can come from its right since particles cannot 
 be placed on the left of their 
 parent and, by definition of $l_n$, there are no particles of generation $n$ in $(-\infty,l_n)$. Once there is survival in $\mathrm{co}(I)$ 
 then if is easy to show, by using a Borel-Cantelli argument, that there is survival in $I$.

\subsection{Results}\label{subsec:mainBPWS}

Throughout this section we consider a BPWS with label measure $\mu$; we denote by $\{m_n\}_{n\in \N}$ and 
$\{m_n^{(2)}\}_{n\in \N}$ the first and second moment of the offspring distribution of the process before selection.
The generating function before selection are denoted by $\{\Phi_n\}_{n \in N}$.

In the following proposition we give a condition for extinction of a BPWS by proving the
absence of an admissible infinite path from the root in the associated accessibility percolation model on the Galton-Watson tree.
This generalizes what was already noted in \cite{cf:CGR14}.

\begin{pro}\label{pro:extinction} Given a BPWS , 
 if there exists $n_0 \ge 0$ such that 
 \[
  \liminf_{n \to +\infty} \frac{\prod_{i=0}^{n-1} m_i}{(n+n_0+1)!}=0,
 \]
 then there is extinction for every starting point $x \in \mathbb{R}$.
\end{pro}
%
%
%

\begin{proof}
We start by supposing that the initial point $x$ is chosen  according to $\mu$.
We use the identification of the BPWS with the associated accessibility percolation model on its infinite Galton-Watson tree:
indeed, if the tree is finite, i.e.~there is extinction before selection, there is extinction also for the BPWS.
Suppose that the Galton-Watson tree $\tau$ is infinite; then, almost surely, the number of leaves at distance $n$ from the root, 
say $s_{n}(\tau)$, has an asymptotic value $s_{n}(\tau) \sim \prod_{i=0}^{n-1} m_i$ as $n \to +\infty$
(use a martingale argument). 
Note that there is a unique
path of length $n$ from the root to each leaf. The probability that a fixed path of length $n$ is admissible is $1/(n+1)!$ since there are $(n+1)!$ possible 
orderings for the $n+1$ labels and all orderings have the same probability. 
Denote by $A_n$ the event ``there exists an admissible path of length $n$ from the root'' and by $\pr_\tau$ the probability conditioned
on the realization $\tau$ of the Galton-Watson tree.
Thus for every $\tau$, $\pr_\tau(A_n) \le s_{n}(\tau) /(n+1)!$. On the other hand, for almost every $\tau$, 
$s_{n}(\tau) /(n+1)! \sim \big ( \prod_{i=0}^{n-1} m_i \big ) /(n+1)!$ as $n \to +\infty$; thus 
$\liminf_{n \to +\infty} s_{n}(\tau) /(n+1)!=0$. Hence,
\[
 \lim_{n \to +\infty} \pr_\tau(A_n) =\liminf_{n \to +\infty} \pr_\tau(A_n) \le \liminf_{n \to +\infty} s_{n}(\tau) /(n+1)!=0.
\]
This yields the result when $n_0=0$.

Suppose $n_0 > 0$ and consider a new BPWS with generating functions $\{\hat \Phi_n\}_{n \in N}$ (before selection) where
$\hat \Phi_n(z):=z$ if $n<n_0$ and $\hat \Phi_n(z):=\Phi_{n-n_0}(z)$ if $n \ge n_0$ (for all $z \in [0,1]$).  This means that 
every particle from generation $0$ to $n_0-1$ has exactly one child. This new BPWS survives with positive probability if and only if
the original one does; indeed, it is enough to note that there is always a positive probability that the unique path from generation $0$ to 
generation $n_0$ is admissible.
The result follows by the first part of the proof by noting that $\prod_{i=0}^{n+n_0-1} \hat m_i=\prod_{i=0}^{n-1} m_i$.

This proves that the probability of extinction is $1$ for almost every starting point $x$ with respect to $\mu$; since
this probability is nondecreasing with respect to the starting point $x$ we have that it is $1$ for all $x \in \mathbb{R}$.
\end{proof}

The interpretation of the previous result in terms of accessibility percolation is the following: given the conditions
of Proposition~\ref{pro:extinction} then for almost every Galton-Watson tree there is no accessibility percolation starting from any label $x$.
The following theorem gives a sufficient condition for survival of a BPWS.


\begin{teo}\label{th:main1}
Suppose that there exists a sequence $\{c_i\}_{i \ge 0}$ of positive real numbers such that $\sum_{i=0}^{+\infty} c_i/m_i <+\infty$ and
\begin{equation}\label{eq:conditionBPWS}
 \begin{cases}
    \sum_{j=n}^{+\infty} \frac{m^{(2)}_j-m_j}{m^2_j} \Big (C^j \prod_{i=n}^{j-1} c_i \Big )^{-1} <+\infty\\
  \inf_{n \in \N} C^n \prod_{j=0}^n c_i >0\\
 \end{cases}
\end{equation}
for some $n \in \N$ and $C>0$. Then the BPWS starting with one particle at $\bar x$ such that
  $\mu(\bar x, +\infty)>0$ survives locally in every $I \subseteq [\bar x, +\infty)$ such that $\mu(I)>0$.
\end{teo}

%
%
%

\begin{proof}

Note that it is enough to prove local survival in $[\bar x, y)$ where 
$\mu(\bar x, y)>0$.
Indeed, if $\mu([y, +\infty) \cap I)>0$ then, according to the Borel-Cantelli lemma
  local survival in $[\bar x, y)$ implies that an infinite number of particles will be placed in 
  $[y, +\infty) \cap I$.
Furthermore, $C$ can always be chosen as equal to 1, by using a new sequence $c^\prime_i:=C c_i$ instead
of $c_i$; thus, 
we assume, without loss of generality, $C=1$. Finally, observe that if the condition~\eqref{eq:conditionBPWS} holds for
some $n=n_1$ then it holds for every $n \ge n_1$.

Fix $\delta \in (0, \mu(\bar x, +\infty))$ and, using the continuity of $\mu$, pick $y$ such that $\mu(\bar x, y)=\delta$.
 Let $n_0 \in \mathbb{N}$ be such that $\sum_{n \ge n_0} c_n/m_n < \delta/2$; $n_0$ can always be chosen larger than $n_1$.
Let $p_n:=\delta/(2n_0)$ for all $n<n_0$ and $p_n:=c_n/m_n$ for all $n\ge n_0$. 
%
 We construct recursively a strictly increasing sequence $\{x_n\}_{n \in \mathbb{N}}$
 satisfying
 \[
  \begin{cases}
  x_0=\bar x,\\
   \mu(x_n,x_{n+1})=p_n.\\   
  \end{cases}
 \]
Clearly $\sum_{n \ge n_0} p_n < \delta/2$ and $\lim_{n \to +\infty} x_n < y$. Indeed
\[
\begin{split}
 \mu(\bar x , \lim_{n \to +\infty} x_n)&= \sum_{n \in \mathbb{N}} \mu(x_n, x_{n+1}) 
=\sum_{n < n_0} p_n +  \sum_{n \ge  n_0} p_n < 
\delta=\mu(\bar x, y).\\
\end{split}
 \]
Thus, if we can prove local survival of the BPWS in $[\bar x , \lim_{n \to +\infty} x_n)$ we have local survival in $[\bar x, y)$.

We proceed by constructing a BPVE which is stochastically dominated by the BPWS as follows: at each generation $n \ge 1$ we obtain a BPVE by
removing all the particles of the BPWS 
outside the interval  $[x_{n-1}, x_{n})$ (along with their progenies). More precisely the BPVE starts with one particle at $\bar x$ which breeds according
to the law of $W_0$ and kills all the particles outside the interval $[x_0, x_1)$; this is equivalent to removing each child independently with probability
$1-p_0$. Given the $n$th generation, we construct the next one by keeping all children of the particles of the $n$th generation which are placed
in the interval $[x_{n}, x_{n+1})$; again, this is like removing each newborn independently with probability $1-p_n$. This is a BPVE which is dominated by the 
original BPWS since if a particle 
is located at $x \in[x_{n-1}, x_{n})$, in the BPWS we keep every child in the interval $[x, +\infty)$ while in the BPVE we keep only those children
which are placed in $[x_{n}, x_{n+1}) \subset [x, y) \subset [x, +\infty)$. Hence, the survival of the BPVE implies the local survival of
the BPWS in $[\bar x, y)$.

Denote by $\widetilde m_n$ and $\widetilde m_n^{(2)}$ the first and second moments respectively of this BPVE. They are related to the moments of the original process:
$\widetilde m_n=p_nm_n$ and $\widetilde m^{(2)}_n-\widetilde m_n=p_n^2(m^{(2)}_n-m_n)$.
Note that $\widetilde m_n=c_n$ and $(\widetilde m^{(2)}_n-\widetilde m_n)/\widetilde m_n 
\big (\prod_{i=n_0}^{n} \widetilde m_i \big )^{-1}=(m^{(2)}_n- m_n)/ m^2_n 
\big (\prod_{i=n_0}^{n-1}  c_i \big )^{-1}$  for all $n\ge n_0$. Theorem~\ref{th:main0} yields the conclusion.
\end{proof}


The following corollary is the analogous of Corollary~\ref{cor:main0} for BPWS, hence we omit the proof
(by definition $\prod_{i=n}^{n-1} c_i:=1$).

\begin{cor}\label{cor:main1}
Suppose that there exists a sequence $\{c_i\}_{i \ge 0}$ of positive real numbers such that $\sum_{i=0}^{+\infty} c_i/m_i <+\infty$ and
one of the following conditions holds for some $C>0$:
 \begin{enumerate}
  \item $\sum_{j=n}^{+\infty} \frac{m^{(2)}_j}{m^2_j} \Big ( C^j \prod_{i=n}^{j-1} c_i \Big )^{-1} <+\infty$ for some $n \in \N$;
  \item $\limsup_{n \to +\infty} \sqrt[n]{m^{(2)}_n/{m^2_n}} < C \liminf_{n \to +\infty} \sqrt[n]{\prod_{i=0}^{n-1} c_i }$;
  \item there exists a function $g:\N\to [1,+\infty)$ such that $m^{(2)}_n/{m^2_n} \le g(n)$ for every sufficiently large $n$ and 
  $\limsup_{n \to +\infty} g(n+1)/g(n) < C \liminf_{n \to +\infty} \sqrt[n]{\prod_{i=0}^{n-1} c_i }$;
    \item $\lim_{n \to +\infty} c_n=+\infty $ and there exists $M,k \ge 1$ such that $m_n^{(2)}/m_n^2 \le k M^n$ for all sufficiently large $n\in \mathbb{N}$;
  \end{enumerate}
then the BPWS starting with one particle at $\bar x$ such that
  $\mu(\bar x, +\infty)>0$ survives locally in every $I \subseteq [\bar x, +\infty)$ such that $\mu(I)>0$. 
\end{cor}

As in Example~\ref{exm:continuous}, explicit examples of laws of $W_n$ satisfying the conditions of Theorem~\ref{th:main1} are: 
\begin{enumerate}
  \item geometric laws: $W_n \sim \mathcal{G}(1/(1+m_n))$ such that $\sum_{i=0}^{+\infty} 1/m_n< +\infty$. 
\item Poisson laws: $W_n \sim \mathcal{P}(m_n)$ where $\sum_{i=0}^{+\infty} 1/m_i<+\infty$;
 \item binomial laws: $W_n \sim \mathcal{B}(k_n, r_n)$ such that $\sum_{i=0}^{+\infty} 1/k_ir_i < +\infty$;
\end{enumerate}
in particular the geometric law corresponds to a continuous-time branching process with selection.

\begin{rem}\label{rem:nalpha}
 Consider a BPWS such that $m_n \sim n^{\alpha}$. For $\alpha < 1$, Proposition~\ref{pro:extinction} holds and there is extinction; 
for $\alpha >1$ (provided that equation~\eqref{eq:conditionBPWS} is satisfied) then by 
Theorem~\ref{th:main1} there is survival. Thus, there is phase transition at the critical exponent $\alpha=1$.

More generally, one can show that (1) if $m_n/(n+\bar n) \le 1$ for all sufficiently large $n$ (and some $\bar n \in \N$) then there is extinction, (2) if
$\liminf m_n/n^\alpha >0$ for some $\alpha >1$ (and equation~\eqref{eq:conditionBPWS} is satisfied)
then there is survival.

\end{rem}

%
%

The role played by 
the sequence $\{c_i\}_{i \ge 0}$ is twofold: on the one hand it allows to treat cases where $\sum_{i=0}^{+\infty} 1/m_i =+\infty$
and, on the other hand, when $\sum_{i=0}^{+\infty} 1/m_i <+\infty$ it allows 
larger upper bounds for ${(m^{(2)}_j-m_j)}/{m^2_j}$. In the following example 
we analyze two explicit cases.

\begin{exmp}\label{exmp:ci}
  Let $\{W_n\}_{n \in \N}$ such that 
 \[
  m_n:=
  \begin{cases}
k+1 & n=2^k\\
(n+1)^2 & \textrm{ otherwise}
  \end{cases}
 \]
 and
  \[
  c_n:=
  \begin{cases}
1/(k+1) & n=2^k\\
2 & \textrm{ otherwise.}
  \end{cases}
 \]
Then $\sum_{n \in \N} 1/m_n=+\infty$ while $\sum_{n \in \N} c_n/m_n <+\infty$. Moreover, if $b \in (1,2)$,
it is easy to prove that $\prod_{i=0}^n c_i = 2^{n-\lfloor \log_2(n) \rfloor}/( \lfloor \log_2(n) \rfloor +1)! \ge b^n$ eventually as $n \to +\infty$;
whence, if $\limsup_{n \to +\infty} \sqrt[n]{m_n^{(2)}/m^2_n} < +\infty$  
then Corollary~\ref{cor:main0}(2) applies and there is survival for the BPWS.
 
 Consider now a process
 where $m_n=b^n$ (for some $b>1$). If $m_n^{(2)}/m_n^2 \le K \alpha^n c^{n(n-1)/2}$ for some $K >0$, $c \in (1,b)$ and $\alpha \ge 1$ then
 Corollary~\ref{cor:main1}(1) applies (with $c_i:= (\alpha+1) c^i$, $C=1$ and $n=0$) and there is survival for the BPWS. 
\end{exmp}

\section*{Acknowledgements}

The authors acknowledge financial support from INDAM-GNAMPA (Istituto Nazionale di Alta Matematica).
The second author also thanks FAPESP (Grant 2015/03868-7) for financial support. Part of this work was carried out during a stay of the second author at  
Laboratoire de Probabilit\'es et Mod\`eles Al\'eatoires, Universit\'e Paris-Diderot, and a visit at Universit\`a di Milano-Bicocca. He is grateful for 
their hospitality and support.

\end{document}